\documentclass{amsart}

\usepackage{amssymb}

\usepackage{paralist}
\usepackage{units}
\usepackage{mathtools}

\newtheorem{theorem}{Theorem}[section]
\newtheorem{lemma}[theorem]{Lemma}
\newtheorem{proposition}[theorem]{Proposition}
\newtheorem{corollary}[theorem]{Corollary}

\theoremstyle{definition}

\theoremstyle{remark}

\numberwithin{equation}{section}

\newcommand{\R}{\mathbb{R}}
\newcommand{\Z}{\mathbb{Z}}
\newcommand{\conv}{\mathrm{conv}}
\newcommand{\half}{\nicefrac{1}{2}}
\newcommand{\mb}{\mathbb}

\begin{document}
\title{Barycenters of points in polytope skeleta}

\author{Michael Gene Dobbins}
\address[MGD]{Dept. Math., Binghamton University, Binghamton, NY 13902, USA}
\email{mdobbins@binghamton.edu}

\author{Florian Frick}
\address[FF]{Dept.\ Math.\ Sciences, Carnegie Mellon University, Pittsburgh, PA 15213, USA}
\email{frick@cmu.edu}

\subjclass[2010]{Primary 51M04, 51M20, 52B11}

\date{\today}

\begin{abstract}
The first author showed that for a given point $p$ in an $nk$-polytope $P$ there are $n$ points in the $k$-faces of~$P$, whose barycenter is~$p$. We show that we can increase the dimension of $P$ by~$r$, if we allow $r$
of the points to be in $(k+1)$-faces. While we can force points with a prescribed barycenter into faces of dimensions $k$ and $k+1$, we show that the gap in dimensions of these faces can never exceed one. We also investigate the weighted analogue of this question, where a convex combination with predetermined coefficients of $n$ points in $k$-faces of an $nk$-polytope is supposed to equal a given target point. While weights that are not all equal may be prescribed for certain values of $n$ and $k$, any coefficient vector that yields a point different from the barycenter cannot be prescribed for fixed $n$ and sufficiently large~$k$.
\end{abstract}

\maketitle

\section{Introduction}

Given an abelian group~$G$ and an integer~${n \ge 2}$, zero-sum problems aim to find sufficient conditions on sequences $x_1, \dots, x_n$ of $n$ elements of $G$ to sum to zero, $x_1+\dots+x_n =0$. The seminal result for this problem area is a theorem of Erd\H os, Ginzburg, and Ziv~\cite{erdos}: any multiset $A \subset \Z/n$ of size $2n-1$ contains $n$ elements (counted with multiplicity) $x_1, \dots, x_n$ such that $x_1+\dots + x_n = 0$. Various generalizations have been established, such as Reiher's proof of the Kemnitz conjecture~\cite{reiher}.

While for finite groups $G$ these problems fall into the realm of combinatorial number theory, they become geometric if $G$ itself has geometry. Here we will study zero-sum problems in Euclidean space, $G = \R^d$. If a set $A \subset \R^d$
is distributed around the origin in a suitable sense, then it contains $n$ vectors that sum to zero---or equivalently, these $n$ vectors have barycenter zero. This was made precise first in~\cite{barba} and then by the first author~\cite{dobbins}: In~\cite{barba} it was shown that if $A$ in $\R^3$ is the $1$-skeleton of a $3$-polytope that contains the origin, then there are $x_1, x_2, x_3 \in A$ with $x_1+x_2+x_3 = 0$. They conjectured that the $1$-skeleton of a polytope in $\R^n$ that contains the origin contains $n$ vectors that sum to zero. More generally, the main result of~\cite{dobbins} establishes that the $k$-skeleton of an $nk$-polytope $P$ with $0 \in P$ contains $n$ vectors that sum to zero. A simplified proof was given in~\cite{blagojevic}.

The proofs in~\cite{dobbins, blagojevic} depend on methods from equivariant topology, and thus crucially make use of the inherent symmetries of the problem. In particular, for a polytope of dimension $d < nk$ with $d > n(k-1)$, the proofs do not generalize to force $nk-d$ of the $x_i$ into faces of dimension~${k-1}$. Our first main result will extend slightly beyond the symmetric case and establish precisely that; see Theorem~\ref{thm:inh-skel-rephr}:

\begin{theorem}
\label{thm:inh-skel}
	Let $P$ be a $d$-polytope with $0 \in P$. Let $n \ge 2$ and $k \ge 0$ be integers such that $d = nk+r$ for some $r \in \{0, \dots, n-1\}$. Then there are points $x_1, \dots, x_{n-r}$ in $k$-faces of~$P$, and $x_{n-r+1}, \dots, x_n$ in $(k+1)$-faces of~$P$ such that $x_1 + \dots + x_n = 0$.
\end{theorem}

We show that this is tight in the sense that dimensions cannot be further decreased (Proposition~\ref{prop}(\ref{dim})), and that the dimensions of faces that the $x_i$ are constrained to cannot differ by more than one (Proposition~\ref{prop}(\ref{F-k})).

Since $\R^d$ is a vector space this gives us another chance to break symmetries and study unbalanced zero-sum problems: Given coefficients $\lambda_1, \dots, \lambda_n > 0$, find sufficient conditions on a set $A \subset \R^d$ to contain $n$ vectors $x_1, \dots, x_n \in A$ with $\lambda_1x_1 + \dots + \lambda_nx_n = 0$. Again, earlier proofs cannot easily be adapted to this asymmetric situation. In fact, we can use Theorem~\ref{thm:inh-skel} to establish results for the case of unbalanced coefficients; see Theorem~\ref{thm:unb-weights}. We also show that if $0 = \lambda_1x_1 + \dots + \lambda_nx_n$ for $x_i$ in the $k$-faces of an $nk$-polytope, then the $\lambda_i$ are almost equal for large~$k$; see Corollary~\ref{cor:balanced}.

\section{Inhomogeneous skeleta}

Let $P^{(k)} \subset \R^{nk}$ denote the $k$-skeleton of a polytope~$P$, that is, the collection of all faces of dimension at most~$k$, and suppose that $P$ contains the origin, then $P^{(k)}$ contains $n$ vectors that sum to zero. This was shown by the first author:

\begin{theorem}[Dobbins~\cite{dobbins}]
\label{thm:skel}
	Let $P \subset \R^{nk}$ be a polytope with $0 \in P$. Then there are $x_1, \dots, x_n \in P^{(k)}$ such that $x_1 + \dots + x_n = 0$.
\end{theorem}

Equivalently, for any given point $p \in P$, where $P$ is an $nk$-polytope, there are points $x_1, \dots, x_n \in P^{(k)}$ with their barycenter $\tfrac1n x_1 + \dots \tfrac1n x_n$ at~$p$. In the sequel, we will use arbitrary target points $p$ in in the polytope, and not only the origin.

Let $\mathcal{P}(d;k_1,\dots,k_n)$ be the predicate ``For any polytope $P$ of dimension at most~$d$, and for any target point $p \in P$, there exist points $x_1,\dots,x_n$ such that $x_i$ is in a $k_i$-face of $P$ and the target point $p$ is the barycenter of the points $x_1,\dots, x_n$.'' With this notation Theorem~\ref{thm:inh-skel} can be rephrased as:

\begin{theorem}
\label{thm:inh-skel-rephr}
	For all $n,d \in \mb{N}$ where $d=nk+r$ for $k \in \mb{Z}$, $r \in \{0, \dots, n-1\}$, the statement 
	$$\mathcal{P}(d;\:\underbrace{k,\dots,\,k}_{n-r},\,\underbrace{k{+}1,\dots,\,k{+}1}_{r})$$ is true.
\end{theorem}

We postpone the proof of Theorem~\ref{thm:inh-skel-rephr} for now. We remark that we cannot decrease the dimension of any $k$-face to~${k-1}$, even if all other $x_i$ may be chosen from faces of arbitrary dimension, including the interior of~$P$ itself. We also have to show that we cannot decrease the dimension of any $(k+1)$-face to~$k$ (again allowing more freedom for the other~$x_i$). We collect these two results here:

\begin{proposition}
\label{prop}
\begin{compactenum}[(a)]
	\item \label{F-k} For $d \ge nk$ the statement $\mathcal{P}(d;\:k-1,\,\underbrace{d,\dots,\,d}_{n-1})$ is false.
	\item \label{dim} For all $n,d \in \mb{N}$ where $d=nk+r$ for $k \in \mb{Z}$, $r \in \{1, \dots, n-1\}$, the statement $\mathcal{P}(d;\:\underbrace{k,\dots,\,k}_{n-r+1},\,\underbrace{d,\dots,\,d}_{r-1})$ is false. 
\end{compactenum}
\end{proposition}

Before proving Theorem~\ref{thm:inh-skel-rephr} and Proposition~\ref{prop}, we first need an additional lemma. In the following denote by $\Delta^d$ the regular $d$-dimensional simplex $$\{(x_1, \dots, x_{d+1}) \in \R^{d+1} \: | \: x_i \ge 0, \sum_i x_i = 1\}.$$ We denote the standard basis of $\R^d$ by $e_1, \dots, e_d$, and the dual basis by $e_1^*, \dots, e_d^*$.

\begin{lemma}
\label{lem-F-r}
    There do not exist points $x_1, \dots, x_n$ in the regular $d$-simplex~$\Delta^d$, $d < n$, where $x_1, \dots, x_{n-d+1}$ are vertices, and the barycenter $\frac{1}{n} \sum_i x_i$ is equal to 
    \[ p = \left(\tfrac{d-\half}{dn},\cdots,\tfrac{d-\half}{dn},\tfrac{n-d+\half}{n}\right) \in \Delta^d. \]
\end{lemma}

\begin{proof}
Suppose there are points $x_1,\dots,x_n \in \Delta^d$ with barycenter~$p$, and that $x_1,\dots,x_{n-d+1}$ are vertices of~$\Delta^d$. We cannot have $x_i = e_{d+1}$ for all $i \in [n-d+1]$, since that would give 
\[ e_{d+1}^*(p) = e_{d+1}^*\left( \tfrac{1}{n} x_1 +\dots+ \tfrac{1}{n} x_{n}  \right) \geq e_{d+1}^*\left( \tfrac{1}{n} x_1 +\dots+ \tfrac{1}{n} x_{n-d+1}  \right) = \tfrac{n-d+1}{n}, \]
but $e_{d+1}^*(p) = \tfrac{n-d+\half}{n} < \tfrac{n-d+1}{n}$. Therefore, at least one of the $x_i$ is a vertex of $\Delta^d$ other than the vertex~$e_{d+1}$. We may assume that $x_1 = e_1$, which gives
\[ e_1^*(p) \geq e_1^*\left(\tfrac{1}{n}x_1\right) = \tfrac{1}{n} ,\]
but $e_1^*(p) = \tfrac{d-\half}{dn} < \tfrac{1}{n}$, which is again a contradiction. Thus, no such points $x_1,\dots,x_n$ exist. 
\end{proof}

\begin{proof}[Proof of Theorem~\ref{thm:inh-skel-rephr}]
For a given $d$-polytope $P$ with $d = nk+r$, $r \in \{0,\dots, n-1\}$, let $Q = P \times \Delta^{s}$ where $s=n-r$, and let
\[ x = \left(\tfrac{s-\half}{sn},\cdots,\tfrac{s-\half}{sn},\tfrac{n-s+\half}{n}\right) \in \Delta^s. \]
Since $Q$ is an $n(k{+}1)$-polytope, by Theorem~\ref{thm:skel} there are points $y_1,\dots,y_n$ in the $(k{+}1)$-faces of $Q$ that have $x$ as their barycenter. Let $x_i \in \mb{R}^{nk+r}$ be the first component of $y_i$ in the product $Q= P \times \Delta$, and let $\tilde y_i$ be the second component of~$y_i$. Then, the $x_i$ are in the $(k{+}1)$-faces of $P$ and sum to zero.

Suppose that at least $r+1$ of the points $x_i$ are not in a $k$-face of~$P$. Then at least $r+1 = n-s+1$ of the points $\tilde y_i$ are vertices of~$\Delta^s$, but that contradicts Lemma~\ref{lem-F-r}, since $x$ is the barycenter of $\{\tilde y_1,\dots,\tilde y_n\}$. Thus, at least $n-r$ of the points $x_i$ are in $k$-faces of~$P$.
\end{proof}

\begin{proof}[Proof of Proposition~\ref{prop}]
\begin{compactenum}[(a)]
\item Let $P = \Delta^{d}$ and let 
\[ p = \left(\tfrac{1}{d+1},\dots,\tfrac{1}{d+1}\right). \]

Suppose there are points $x_1,\dots,x_n \in \Delta^d$ with barycenter $p$, and that one of the points is in a $(k{-}1)$-face of $\Delta^d$. We may assume $x_1 \in \conv\{e_1,\dots,e_k\}$. Let $\phi = e_1^* + \dots +e_k^*$.  Then, we have
\[ \phi(p) = \phi\left(\tfrac{1}{n}x_1+\dots +\tfrac{1}{n}x_n\right) \geq \phi\left(\tfrac{1}{n}x_1\right) = \tfrac{1}{n}, \]
but this is a contradiction, since $\phi(p) = \tfrac{k}{d+1} \leq \tfrac{k}{nk+1} < \tfrac{1}{n}$. Thus, no such points $x_1,\dots,x_n$ exist.

\item Let 
\[ P = \underbrace{\Delta^n \times \dots \times \Delta^n }_{k} \times \Delta^r \]
\[ 
\begin{array}{r@{\ }c@{\ }l}
p
&=& p_1 \times \dots \times p_{k+1} \vspace{3pt} \\
&=& \left(\mathrlap{\phantom{\tfrac{1}{n+1}}}\right.\!\underbrace{\tfrac{1}{n+1},\dots,\,\tfrac{1}{n+1}}_{k(n+1)},\,\underbrace{\tfrac{r-\half}{rn},\cdots,\, \tfrac{r-\half}{rn}}_r,\,\left.\tfrac{n-r+\half}{n}\right) \\
\end{array} \]
where $p_1,\dots,p_{k} \in \Delta^n$ are each the target point from part~(\ref{F-k}) when $k=1,r=0$, and $p_{k+1} \in \Delta^r$ is the target point from Lemma~\ref{lem-F-r}.

Suppose there are points $x_1,\dots,x_n \in P$ with barycenter $p$, and that the points $x_1,\dots,x_{n-r+1}$ are in $k$-faces of~$P$. Let $x_{i,j}$ be the $j$-th component of $x_i$ from the product defining $P$ above, and let $k_{i,j}$ be the dimension of the minimal face containing $x_{i,j}$. That is, for $j \in \{1,\dots,k\}$, $x_{i,j}$ is given by the $j$-th block of $n+1$ consecutive coordinates of~$x_i$, 
\[ x_{i,j} = (e_{(n+1)j -n}^*(x_i),\dots,e_{(n+1)j}^*(x_i)) \in \Delta^n \]
And, $x_{i,k+1} $ is given by the last $r+1$ coordinates of~$x_i$, 
\[ x_{i,k+1} = (e_{(n+1)k +1}^*(x_i),\dots,e_{(n+1)k +r +1}^*(x_i)) \in \Delta^r. \]

Since the barycenter of $\{x_{1,j},\dots,x_{n,j}\}$ for $j \in \{1,\dots,k\}$ is 
\[p_j = \left(\tfrac{1}{n+1},\cdots,\tfrac{1}{n+1}\right), \]
by part~(\ref{F-k}), none of the points $x_{i,j}$ are vertices of $\Delta^n$. Hence, $k_{i,j} \geq 1$ for $j \in \{1,\dots,k\}$, so ${k_{i,1}+\dots+k_{i,k}} \geq k$. Since $x_{i}$ is in a $k$-face for $i \in \{1,\dots,n-r+1\}$, we must have $k_{i,1}+\dots+k_{i,k+1} \leq k$, so $k_{i,k+1} = 0$. That is, $n-r+1$ of the points $x_{i,k+1}$ must be vertices of $\Delta^r$, but that is impossible since the barycenter of $\{x_{1,k+1},\dots,x_{n,k+1}\}$ is
\[p_{k+1} = \left(\tfrac{r-\half}{rn},\cdots,\tfrac{r-\half}{rn},\tfrac{n-r+\half}{n}\right), \]
so by Lemma~\ref{lem-F-r} at most $n-r$ of the points $x_{i,k+1}$ can be vertices of $\Delta^r$. Thus, no such points $x_1,\dots,x_n$ exist. 
\end{compactenum}
\end{proof}

\section{Unbalanced weights}

Given positive integers $n$ and $k$, what are all $n$-tuples $(\lambda_1, \dots, \lambda_n)$ of coefficients $\lambda_i > 0$ normalized to $\lambda_1 + \dots + \lambda_n = 1$ such that for any $nk$-polytope $P \subset \R^{nk}$ with $0 \in P$ there are $x_1, \dots, x_n \in P^{(k)}$ with $\lambda_1x_1 + \dots + \lambda_nx_n = 0$? We denote the set of all such coefficients with $\lambda_1 \ge \lambda_2 \ge \dots \ge \lambda_n$ by~$\Lambda(n,k)$. The set $\Lambda(n,k)$ is nonempty since it contains $(\frac1n, \dots, \frac1n)$ by Theorem~\ref{thm:skel}. By taking $P$ to be closer and closer approximations of the unit ball in~$\R^{2k}$, we see that $\Lambda(2,k) = \{(\frac12, \frac12)\}$. We will show that $\Lambda(n,1)$ may contain more than one element for $n > 2$.

\begin{theorem}
\label{thm:unb-weights}
	Let $s$, $t$, and $k$ be positive integers, and let $n = sk+t(k+1)$. Then the coefficient vector $(\lambda_1, \dots, \lambda_n)$ with $\lambda_1 = \dots = \lambda_{sk} = \frac{1}{(s+t)k}$ and $\lambda_{sk+1} = \dots = \lambda_n = \frac{1}{(s+t)(k+1)}$ is contained in~$\Lambda(n,1)$.
\end{theorem}

\begin{proof}
	Let $P$ be an $n$-polytope with $0 \in P$. By Theorem~\ref{thm:inh-skel-rephr} there are points $x_1, \dots, x_s$ in $k$-faces of~$P$ and points $x_{s+1}, \dots, x_{s+t}$ in $(k+1)$-faces of~$P$, such that $\sum x_i = 0$. Each $x_i$ with $i \in \{1, \dots, s\}$ can be written as $x_i = \frac{1}{k}\sum_j y_j^{(i)}$ for $y_1^{(i)}, \dots, y_k^{(i)} \in P^{(1)}$. Similarly, each $x_i$ with $i \in \{s+1, \dots, s+t\}$ can be written as $x_i = \frac{1}{k+1}\sum_j y_j^{(i)}$ for $y_1^{(i)}, \dots, y_{k+1}^{(i)} \in P^{(1)}$. Putting this together, we obtain
	$$0 = \sum_{i=1}^s \frac{1}{k}\sum_j y_j^{(i)} + \sum_{i=s+1}^{s+t} \frac{1}{k+1}\sum_j y_j^{(i)}.$$
	Normalizing these coefficients to sum up to one, we get the result.
\end{proof}

This shows that unbalanced weights may be prescribed for certain parameters. Our last goal is to show that asymptotically, that is, for fixed~$n$ and large~$k$, unbalanced weights may not be prescribed. More precisely,
given $n$ positive real numbers $\lambda_1, \dots, \lambda_n$ with $\lambda_1 + \dots + \lambda_n = 1$ that are not all equal to~$\frac1n$, there is an integer $k$ and an $nk$-polytope $P$ with $0 \in P$ such that $\lambda_1x_1 + \dots + \lambda_nx_n \ne 0$ for all $x_1, \dots, x_n \in P^{(k)}$. We will need the following simple lemma:

\begin{lemma}
\label{lem:coeff-bound}
	Let $k \ge 1$ and $n \ge 2$ be integers. Let $d = nk$ and $x_1, x_2 \in \Delta^d$, where $x_1$ is contained in a $k$-face of~$\Delta^d$. Suppose $\lambda_1x_1+\lambda_2x_2 = (\frac{1}{d+1}, \dots, \frac{1}{d+1})$ for $\lambda_1,\lambda_2 \ge 0$ with $\lambda_1+\lambda_2=1$. Then $\lambda_1 \le \frac{k+1}{d+1}$.
\end{lemma}

\begin{proof}
	Let $p = (\frac{1}{d+1}, \dots, \frac{1}{d+1})$. We may assume $x_1 \in \conv\{e_1,\dots,e_{k+1}\}$. Let $\phi = e_1^* + \dots +e_{k+1}^*$.  Then, we have
    \[ \tfrac{k+1}{nk+1} = \tfrac{k+1}{d+1} = \phi(p) = \phi\left(\lambda_1x_1+\lambda_2x_2\right) \ge \phi\left(\lambda_1x_1\right) = \lambda_1. \]
\end{proof}

\begin{theorem}
	Let $n>0$ and $k>0$ be integers. If $(\lambda_1, \dots, \lambda_n) \in \Lambda(n,k)$ then $\lambda_i \le \frac{k+1}{nk+1}$ for all~$i$.
\end{theorem}

\begin{proof}
	Let $d = nk$. Translate the standard simplex $\Delta^d$ such that its barycenter is at the origin. Now suppose 
	\[ 0 = \lambda_1x_1 + \dots + \lambda_nx_n, \]
	where the $x_i$ are in $k$-faces of~$\Delta^d$, and the coefficients $\lambda_i$ are nonnegative and satisfy $\lambda_1+\dots+\lambda_n = 1$. We also assume $\lambda_1 \ge \lambda_2 \ge \dots \ge \lambda_n$.
	Then $\lambda_2 + \dots + \lambda_n > 0$ and thus 
	\[ 0 = \lambda _1x_1 + (\lambda_2 + \dots + \lambda_n)(\tfrac{\lambda_2}{\lambda_2 + \dots + \lambda_n}x_2 
	+ \dots + \tfrac{\lambda_n}{\lambda_2 + \dots + \lambda_n}x_n). \]
	Since $\tfrac{\lambda_2}{\lambda_2 + \dots + \lambda_n}x_2 + \dots + \tfrac{\lambda_n}{\lambda_2 + \dots + \lambda_n}x_n$ is a convex combination of points in~$\Delta^d$, it is itself a point in~$\Delta^d$. Then by Lemma~\ref{lem:coeff-bound}, we have that $\lambda_1 \le \frac{k+1}{d+1} = \frac{k+1}{nk+1}$.
\end{proof}

It is a simple consequence that $\bigcap_k \Lambda(n,k) = \{(\tfrac1n, \dots, \tfrac1n)\}$, or in different words:

\begin{corollary}
\label{cor:balanced}
	Let $\lambda=(\lambda_1, \dots, \lambda_n)$ be some unbalanced coefficient vector, that is, $\lambda_i \ge 0$ for all~$i$, $\lambda_1 + \dots + \lambda_n = 1$, and $\lambda \ne (\tfrac1n, \dots, \tfrac1n)$. Then for $k$ sufficiently large $\lambda \notin \Lambda(n,k)$.
\end{corollary}

\providecommand{\MRhref}[2]{%
  \href{http://www.ams.org/mathscinet-getitem?mr=#1}{#2}
}
\providecommand{\href}[2]{#2}

\end{document}